\documentclass{amsart}
\usepackage{graphicx} 
\usepackage{amsmath,amssymb,amsthm,mathtools,xpatch}
\usepackage{hyperref}
\usepackage{xcolor}
\usepackage{a4wide}
\usepackage{latexsym}

\newtheorem{theorem}{Theorem}[section]

\newtheorem{remark}[theorem]{Remark}
\newtheorem{definition}[theorem]{Definition}

\newtheorem{lemma}[theorem]{Lemma}

\makeatletter
\renewcommand{\email}[1]{\thanks{Email addresses: \texttt{#1}}}
\makeatother

\begin{document}

\title[Improved Lipschitz Retraction Bounds]{Improved Bounds For The Lipschitz Retraction Constant In Classical Banach Spaces}
\author[Shriram Mahesh]{Shriram Mahesh \\ St John's College, University of Cambridge} 
\begin{abstract}
    We improve the upper and lower bounds for the smallest Lipschitz constant $K$ for which there exists a $K$-Lipschitz retraction from the unit ball to the unit sphere of the infinite-dimensional Banach space $X$, for several classical spaces $X$.
\end{abstract}
\email{sm2803@cam.ac.uk, shriram.mahesh@gmail.com}
\date{}

\maketitle

\section{Introduction}
\noindent Let $X$ be a Banach space, and let $B_X$ and $S_X$ denote the unit ball of $X$ and the unit sphere of $X$ respectively. In the case where $X$ is finite-dimensional, one of the equivalent forms of the Brouwer fixed point theorem states that there exists no continuous retraction $R : B_X \to S_X$ (where a retraction is a function fixing all points in $S_X$).
\newline
\newline
The situation is different in infinite dimensions, as there exist Lipschitz retractions from the ball to the sphere \cite{benyamini1983spheres}. In fact there exists a universal constant $C > 0$ such that for all infinite-dimensional Banach spaces $X$, there exists a $C$-Lipschitz retraction $R: B_X \to S_X$ (i.e. a retraction with $||R(x) - R(y)|| \leq C||x-y||$ $\forall x, y \in B_X$). See \cite{goebel1990topics} for a proof of this result. This leads us to define the natural quantity 
\[
k_0(X) := \inf\{k>0: \exists \hspace{0.3em} k\text{-Lipschitz retraction in X} \}
\]
Surprisingly, $k_0(X)$ is not known exactly for any Banach space $X$. However, many upper and lower bounds have been proven. For instance, it is known that:
\begin{align*}
& k_0(X) \geq 3 \qquad \text{for all Banach spaces $X$ (see \cite{goebel1990topics})}\\
& k_0(X) > 3 \qquad \text{for all uniformly convex Banach spaces $X$ (see \cite{piasecki2011retracting})}\\
& 4 \leq k_0(l_1) \leq 8 \qquad \text{(see \cite{annoni2007upper} and the third reference therein)}\\
& 4.58 \leq k_0(l_2) \leq 28.99 \qquad \text{(see \cite{casini2017minimal} and \cite{baronti2003retraction})}\\
& k_0(l_\infty) \leq 12 + 2\sqrt{30} \qquad \text{(see \cite{casini2017minimal})}\\
& k_0(X) \leq 4(2 + \sqrt{3}) \qquad \text{for $X \in \{c, c_0, C[0,1]\}$ (see \cite{piasecki2011retracting})}   
\end{align*}
The cited references also contain other results along the same lines. Interestingly, no lower bound for $k_0(l_p)$ in the case $1 < p < 2$ is recorded in the literature beyond the general bound for uniformly convex spaces $k_0(l_p) > 3$. In this work, we improve the lower bounds for $k_0(l_p)$ for small $p$ (including $p=1$) and $k_0(c_0)$. We also improve the upper bounds for $k_0(c_0)$, $k_0(l_1)$ and $k_0(l_\infty)$.
\newpage

\section{Improved Lower Bounds}
\noindent In this section, we first prove an improvement to the best known lower bound for $k_0(l_p)$ when $p$ is small. We will require the following (very easy) lemma:

\begin{lemma}
Let $1 \leq p < \infty$. Let $(e_n)_{n=1}^\infty$ denote the standard basis of $l_p$. Let $A \subset \mathbb{N}$, and define $V := span\{e_k: k \in A\}$, $V^\perp := span\{e_k: k \not\in A\}$. Let $P_V$, $P_{V^\perp}$ denote the projections from $l_p$ onto $V$ and $V^\perp$ respectively. Suppose $x, y \in S_{l_p}$ are such that $||P_Vx||^p = a$ (hence $||P_{V^\perp}x||^p = 1 - a$) and $||P_Vy||^p = b$ (hence $||P_{V^\perp}x||^p = 1 - b$). Then $||x - y|| \geq (|a^\frac{1}{p} - b^\frac{1}{p}|^p + |(1-a)^\frac{1}{p} - (1-b)^\frac{1}{p}|^p)^\frac{1}{p}$.
\end{lemma}

\begin{proof}
We have:
\begin{align*}
||x-y||^p &= \sum_{i \in A} |x_i - y_i|^p + \sum_{i \not\in A} |x_i - y_i|^p \\ &\geq \left|(\sum_{i \in A} |x_i|^p)^{\frac{1}{p}} - (\sum_{i \in A}|y_i|^p)^{\frac{1}{p}}\right|^p + \left|(\sum_{i \not\in A} |x_i|^p)^{\frac{1}{p}} - (\sum_{i \not\in A}|y_i|^p)^{\frac{1}{p}}\right|^p \\ &= |a^\frac{1}{p} - b^\frac{1}{p}|^p + |(1-a)^\frac{1}{p} - (1-b)^\frac{1}{p}|^p,
\end{align*}
where the second inequality follows from Minkowski's inequality applied twice. Taking $p^{th}$ roots gives us the desired result.
\end{proof}

\begin{theorem}
Let $1 \leq p < \infty$. Then $k_0(l_p) \geq \kappa_p$, for a constant $\kappa_p$ to be specified, with $\kappa_p > 3$ for $p \leq 1.6$. Also, we have in particular that $k_0(l_1) \geq 3 + \sqrt{5} \approx 5.236$.
\end{theorem}

\begin{proof}
Let $R: B_{l_p} \to S_{l_p}$ be a $K$-Lipschitz retraction and fix $\epsilon > 0$. We may find a finite-dimensional subspace $E$ of $l_p$ which is spanned by a finite subset of the standard basis such that $||P_ER(0)||^p \geq 1 - \epsilon$ (where $P_E$ is, as in the lemma, the (continuous) projection onto $E$). Say $E = span\{e_k: k \in I\}$, where $I \subset \mathbb{N}$ is some finite set. Consider the map $T: B_E \to B_E$ mapping $x \mapsto P_ER(x)$. This is continuous, by continuity of $R$ and $P_E$, and if $x \in S_E$ then, because $R$ fixes unit vectors, $T(x) = x$. Hence the map $T$ must have a root, else $x \mapsto \frac{T(x)}{||T(x)||}$ would be a continuous retraction from $B_E \to S_E$, contradicting Brouwer's fixed point theorem (because $E$ is finite-dimensional). Call this root $x_0$, and let $r := ||x_0||$.
\newline
\newline
Observe that we thus have $||P_ER(x_0)||^p = 0$ and $||P_ER(0)||^p \geq 1-\epsilon$, so Lemma $2.1$, combined with the Lipschitz condition, yields
\[
\eta_\epsilon := ( (1-\epsilon) + (1 - \epsilon^\frac{1}{p})^p)^\frac{1}{p} \leq ||R(x_0) - R(0)|| \leq K||x_0|| = Kr,
\]
so $r \geq \frac{\eta_\epsilon}{K}$.
\newline
\newline
Now pick another finite $J \subset \mathbb{N} \backslash I$ such that, for $F := span\{e_k: k \in J\}$, $||P_FR(x_0)||^p \geq 1 - \epsilon$. Consider the map $S: B_F \to B_F$ which maps $y \mapsto P_FR(x_0 + (1-r^p)^{\frac{1}{p}}y)$. Once more, this map is continuous, and if $y \in S_F$, then
\[
||x_0 + (1-r^p)^{\frac{1}{p}}y||^p = ||x_0||^p + (1-r^p)||y||^p = r^p + 1 - r^p = 1,
\]
where the first equality follows from the fact that $x_0 \in E$, $y_0 \in F$ and $E \cap F = \{0\}$. Thus for $y \in S_F$, we have $S(y) = (1-r^p)^{\frac{1}{p}}y$, again because $R$ fixes unit vectors. So, by the same reasoning as before, $S$ must have a root $y_0$, else $y \mapsto \frac{S(y)}{||S(y)||}$ would be a continuous retraction in $F$, contradicting Brouwer. Defining $z := x_0 + (1-r^p)^\frac{1}{p}y_0$, we thus have $||P_FR(z)||^p = 0$ and $||P_FR(x_0)||^p \geq 1-\epsilon$. By Lemma $2.1$ and the Lipschitz condition, 
\begin{align*}
\eta_\epsilon \leq ||R(z) - R(x_0)|| \leq K||(1-r^p)^\frac{1}{p}y_0||,
\end{align*}
hence $||y_0|| \geq \frac{\eta_\epsilon}{K (1-r^p)^\frac{1}{p}}$. We will proceed by finding an upper bound for $||y_0||$ and comparing our bounds. To do this, write $||y_0|| = 1 - \delta$. Then $\frac{y_0}{1-\delta} \in S_F$, so $z_\delta := x_0 + (1-r^p)^\frac{1}{p}\frac{y_0}{1-\delta} \in S_{l_p}$. So $R$ fixes this quantity, meaning that $||P_FR(x_0 + (1-r^p)^\frac{1}{p}\frac{y_0}{1-\delta})||^p  = 1-r^p$. Since $||P_FR(x_0 + (1-r^p)^{\frac{1}{p}}y_0)||^p = 0$, we once more apply Lemma $2.1$ and the Lipschitz condition to obtain: 
\begin{align*}
(1 - r^p + (1-r)^p)^\frac{1}{p} \leq ||R(z_\delta) - R(z)|| &\leq K \left|\left|(1-r^p)^\frac{1}{p}\frac{y_0}{1-\delta} - (1-r^p)^\frac{1}{p}y_0\right|\right| \\ &=K (1-r^p)^{\frac{1}{p}} \delta,
\end{align*}
hence $\delta \geq (\frac{1 - r^p + (1-r)^p}{K^p (1-r^p)})^\frac{1}{p}$, and so $||y_0|| = 1 - \delta \leq 1 - (\frac{1 - r^p + (1-r)^p}{K^p (1-r^p)})^\frac{1}{p}$. Comparing bounds for $||y_0||$, and using the previously derived lower bound for $r$, we obtain that:
\begin{align*}
\frac{\eta_\epsilon}{K} &\leq (1-r^p)^{\frac{1}{p}}\left(1 - \left(\frac{1 - r^p + (1-r)^p}{K^p (1-r^p)}\right)^\frac{1}{p}\right) \\ &\leq \left(1-\frac{\eta_\epsilon^p}{K^p}\right)^{\frac{1}{p}}\left(1 - \left(\frac{1 - r^p + (1-r)^p}{K^p (1-r^p)}\right)^\frac{1}{p}\right) \qquad \qquad (*)
\end{align*}
In the case $p = 1$, upon taking $\epsilon \to 0$ (as $\epsilon$ was arbitrary), and noting that $\eta_\epsilon \to 2$ in the limit, we obtain:
\[
\frac{2}{K} \leq \left(1 - \frac{2}{K}\right)^2
\]
Solving this yields that either $\frac{2}{K} \geq \frac{3 + \sqrt{5}}{2}$ or $\frac{2}{K} \leq \frac{3 - \sqrt{5}}{2}$. The general Banach space bound $k_0(X) \geq 3$ gives $K \geq 3$, hence $\frac{2}{K} \leq \frac{2}{3}$, so only the second option is possible, and rearranging this option yields $K \geq 3 + \sqrt{5}$. Since this applies for all Lipschitz retractions in $l_1$, we obtain the desired $k_0(l_1) \geq 3 + \sqrt{5}$.
\newline
\newline
For $p >1$, the expression in the second pair of brackets doesn't simplify as nicely, so we just bound it above by $(1 - \frac{1}{K})$. Then taking $\epsilon \to 0$ once more,
\[
\frac{2^\frac{1}{p}}{K} \leq \left(1 - \frac{2}{K^p}\right)^\frac{1}{p}\left(1 - \frac{1}{K}\right)
\]
Analogously to the case $p=1$, this yields $K \geq \kappa_p$, where $\kappa_p$ is defined as the largest root to this equation. This bound beats $3$ for $p < p^*$, where $p^*$ is such that:
\[
\frac{2^\frac{1}{p}}{3} = \left(1 - \frac{2}{3^p}\right)^\frac{1}{p}\left(1 - \frac{1}{3}\right)
\]
Solving this yields $p^* = 1.60552...$, and the proof is complete.
\end{proof}

\noindent Next, we prove an improved lower bound for $k_0(c_0)$. We will require a definition and a lemma:

\begin{definition}
    Let $\gamma: [0, 1] \to S_{c_0}$ be a continuous function (i.e. a curve). Then we define the \textbf{length} of $\gamma$ to be:
    \[
    L(\gamma) := \sup_{\substack{n \in \mathbb{N} \\ 0 = t_0 < t_1 < \cdots < t_n = 1}} \sum_{i = 1}^n ||\gamma(t_i) - \gamma(t_{i-1})||
    \]
\end{definition}

\begin{lemma}
Let $\gamma: [0, 1] \to S_{c_0}$ be a continuous curve, and suppose that there exists a subspace $F \subset c_0$, spanned by a subset of the standard basis, such that:
\begin{align*}
    &||P_{F^\perp}\gamma(0)|| = 1, \quad ||P_F\gamma(0)|| = a < 1 \\
    &||P_{F^\perp}\gamma(t^*)|| = 0, \quad ||P_F\gamma(t^*)|| = 1 \\
    &||P_{F^\perp}\gamma(1)|| = 1, \quad ||P_F\gamma(1)|| = 0
\end{align*}
for some $t^* \in (0, 1)$. Then $L(\gamma) \geq 4-a$.
\end{lemma}

\begin{proof}
Define $t^{**} := \inf\{t: P_F\gamma(t) = 1\} \in [0, t^*]$. For $t < t^{**}$, $||P_F\gamma(t)|| < 1$, so because $\gamma$ has codomain $S_{c_0}$, $||P_{F^\perp}\gamma(t)||=1$. By continuity, we have $||P_F\gamma(t^{**})|| = ||P_{F^\perp}\gamma(t^{**})|| = 1$. Similarly, we may find $t^{***} \in [t^*, 1]$ such that $||P_F\gamma(t^{***})|| = ||P_{F^\perp}\gamma(t^{***})|| = 1$. Since $||a - b|| \geq \max\{||P_Fa - P_Fb||, ||P_{F^\perp}a - P_{F^\perp}b||\}$ $\forall a, b \in c_0$, we thus have:
\begin{align*}
    L(\gamma) &\geq ||\gamma(t^{**}) - \gamma(0)|| + ||\gamma(t^*) - \gamma(t^{**})|| + ||\gamma(t^{***}) - \gamma(t^*)|| + ||\gamma(1) - \gamma(t^{***})|| \\ &\geq ||P_F\gamma(t^{**}) - P_F\gamma(0)|| + ||P_{F^\perp}\gamma(t^*) - P_{F^\perp}\gamma(t^{**})|| + \\ &||P_{F^\perp}\gamma(t^{***}) - P_{F^\perp}\gamma(t^*)|| + ||P_F\gamma(1) - P_F\gamma(t^{***})|| \\&\geq 1-a + 1 + 1 + 1 \\ &= 4-a,
\end{align*}
as desired.
\end{proof}

\begin{theorem}
$k_0(c_0) \geq 4$.
\end{theorem}

\begin{proof}
Let $R: B_{c_0} \to S_{c_0}$ be a $K$-Lipschitz retraction - it suffices to show $K \geq 4$. Fix $\epsilon > 0$. Note that by the decay of sequences in $c_0$, every element in this space achieves its (sup-)norm. Hence, if $P_N$ is the projection in $c_0$ onto the first $N$ coordinates, and $Q_N := id - P_N$, we may pick $N$ large enough that $||P_NR(0)|| = 1$, $||Q_NR(0)|| < \epsilon$. Let $E_N$ be the finite-dimensional subspace consisting of sequences supported in the first $N$ coordinates. Then consider the map $T: B_{E_N} \to B_{E_N}$ mapping $x \mapsto P_NR(x)$. This is continuous, hence by Brouwer has a root $x_0$, else $x \mapsto \frac{T(x)}{||T(x)||}$ would be a continuous retraction $B_{E_N} \to S_{E_N}$.
\newline
\newline
Now consider the path $\gamma: [0, 1] \to S_{c_0}$ mapping $t \mapsto R(t \cdot \frac{x_0}{||x_0||})$. We have:
\begin{align*}
    &||P_N\gamma(0)|| = 1, \quad ||Q_N\gamma(0)|| < \epsilon \\
    &||P_N\gamma(||x_0||)|| = 0, \quad ||Q_N\gamma(||x_0||)|| = 1 \\
    &||P_N\gamma(1)|| = 1, \quad ||Q_N\gamma(1)|| = 0
\end{align*}
By Lemma 2.4, this implies that the length of $\gamma$ is at least $4 - \epsilon$. But note that $||\gamma(t) - \gamma(t')|| = ||R(t \cdot \frac{x_0}{||x_0||}) - R(t' \cdot \frac{x_0}{||x_0||})|| \leq K||t \cdot \frac{x_0}{||x_0||} - t' \cdot \frac{x_0}{||x_0||}|| = K|t - t'|$, hence $\gamma$ is also $K$-Lipschitz, so its length is at most $K(1-0) = K$. Thus $K \geq 4 - \epsilon$, so by arbitrariness of $\epsilon$, $K \geq 4$, which completes the proof.
\end{proof}
\newpage

\section{Improved Upper Bounds}
\noindent In this section, we provide $3$ constructions of Lipschitz retractions, one in each of the spaces $c_0, l_1, l_\infty$, which have smaller Lipschitz constants than any such maps previously constructed. We will first make a definition:

\begin{definition}
For $x \in l_\infty$ and $r > 0$, define the coordinate-wise truncation $Q_r(x) \in rB_{l_\infty}$ by $Q_r(x) := ([x_1]_r, [x_2]_r, \cdots)$, where 
\[
[x_j]_r := \begin{cases}
    r & \text{if} \quad r \leq x_j \\
    x_j & \text{if} \quad -r \leq x_j \leq r \\
    -r & \text{if} \quad x_j \leq -r
\end{cases}
\]
\end{definition}
\noindent Observe that for fixed $r$, $|[x_j]_r - [y_j]_r| \leq |x_j - y_j|$ for all $j \in \mathbb{N}$, hence $||Q_r(x) - Q_r(y)||_\infty \leq ||x-y||_\infty$, so $Q_r$ is non-expansive (i.e. $1$-Lipschitz). In fact, we easily see that $||Q_s(x) - Q_r(y)|| \leq \max\{||x-y||, |s-r|\}$
\newline
\newline
We now turn to the improved upper bound for $k_0(l_\infty)$. In fact, this improvement follows from only a small modification to the method of Casini-Piasecki \cite{casini2017minimal}, and we do not claim any originality apart from changing the parameters in this approach. However, since the tweak yields a sizeable improvement of around $2.5$ to the upper bound, which seems not to be found in the literature, we have decided to rewrite the argument in full.

\begin{theorem}
$k_0(l_\infty) \leq 20.5$.
\end{theorem}

\begin{proof}
Let 
\[
\psi(k) := \sup_{\substack{T: B_{l_\infty} \to B_{l_\infty} \\ Lip(T) \leq k}} \inf_{x \in B_{l_\infty}} ||x - Tx||
\]
be the minimal displacement characteristic. It is well-known that for $k > 1$, $\psi(k) > 0$ (this holds for any Banach space, not just $l_\infty$), so let $k > 1$, and pick $0 < \delta < \psi(k)$ and a $k$-Lipschitz map $T: B_{l_\infty} \to B_{l_\infty}$ with $||x - Tx|| \geq \delta$ $\forall x \in B_{l_\infty}$. Define $b := 1 + \delta + \frac{1}{k}$, and let $\rho$ be the larger root of:
\[
(k+1-\frac{k}{b}\rho)\rho = \delta
\]
Then note that $\rho > b$ - indeed, this is because for $g(r) := (k+1-\frac{k}{b}r)r$, we have $g(b) = b > \delta$. Now define $F: \rho B_{l_\infty} \to B_{l_\infty}$ by
\[
F(x) := \begin{cases}
    x - Tx & \text{if} \quad ||x|| \leq 1 \\
    x - T(Q_1(x)) & \text{if} \quad 1 \leq ||x|| \leq 1 + \delta \\
    x - Q_{k(b-||x||)}(T(Q_1(x))) & \text{if} \quad 1+ \delta \leq ||x|| \leq b \\
    (k+1-\frac{k}{b}||x||)x & \text{if} \quad b \leq ||x|| \leq \rho
\end{cases}
\]
Note that this is well-defined - to see this, we check that the function is consistently defined at $||x|| \in \{1, 1+\delta, b\}$. If $||x|| = 1$, then $Q_1(x) = x$, which immediately yields consistency at $||x|| = 1$. If $||x|| = 1+\delta$, $k(b - ||x||) = k \cdot \frac{1}{k} = 1$, hence 
\[
 x - Q_{k(b-||x||)}(T(Q_1(x))) = x - Q_1(T(Q_1(x))) = x - T(Q_1(x)),
\]
because the codomain of $T$ is $B_{l_\infty}$, yielding consistency at $||x|| = 1+\delta$. Finally, if $||x|| = b$, then $k(b-||x||) = 0$, so 
\[
x - Q_{k(b-||x||)}(T(Q_1(x))) = x - Q_0(T(Q_1(x))) = x = (k+1-\frac{k}{b}||x||)x,
\]
which completes the verification of consistency.
\newline
\newline
Furthermore, from the fact that $T$ is $k$-Lipschitz and the bound $||Q_s(x) - Q_r(y)|| \leq \max\{||x-y||, |s-r|\}$, we have that $F$ is $(k+1)$-Lipschitz (we may verify this on each piece on which the piecewise function $F$ is defined and extend this to $F$ as a whole). Now, we prove that $||F(x)|| \geq \delta$ $\forall x \in \rho B_{l_\infty}$. For $||x|| \leq 1$, this follows immediately from the displacement property of $T$. For $||x|| \in [1, 1+\delta]$,
\[
||x - T(Q_1(x))|| \geq ||Q_1(x) - Q_1(T(Q_1(x)))|| = ||Q_1(x) - T(Q_1(x))|| \geq \delta,
\]
by non-expansiveness of $Q_1$ and the displacement property of $T$. For $||x|| \in [1+\delta, b]$, the vector being subtracted from $x$ has norm at most $1$, giving the desired lower bound, and for $||x|| \in [b, \rho]$, this follows from the fact that $(k+1-\frac{k}{b}||x||)||x|| \geq (k+1-\frac{k}{b} \rho)\rho = \delta$.
\newline
\newline
We may now construct the retraction $R: B_{l_\infty} \to S_{l_\infty}$ mapping $x \mapsto Q_1(\frac{F(\rho x)}{\delta})$. By the lower bound on $||F(x)||$, we know that $||\frac{F(\rho x)}{\delta}|| \geq 1$, hence $R(x)$ really does have norm $1$ (i.e. map into $S_{l_\infty}$) $\forall x \in B_{l_\infty}$. Furthermore, for $||x|| = 1$, $||\rho x|| = \rho$, hence $F(\rho x) = \rho x$. Hence $R(x) = x$, so $R$ does indeed fix the points of $S_{l_\infty}$, so is indeed a retraction.
\newline
\newline
Finally, $R(x)$ is clearly $\frac{\rho (k+1)}{\delta}$-Lipschitz. Using the quadratic formula, 
\[
\rho = \frac{b}{2k}\left[k+1 + \sqrt{(k+1)^2 - \frac{4k \delta}{b}}\right] = \frac{1 + \delta + \frac{1}{k}}{2k}\left[k+1 + \sqrt{(k+1)^2 - \frac{4k \delta}{1 + \delta + \frac{1}{k}}}\right]
\]
Bolibok \cite{bolibok2012minimal} proved that $\psi(k) \geq 1 - \frac{2}{k}$ for $k > 2 + \sqrt{2}$, so with this restriction on $k$, as $\delta \in (0, \psi(k))$ was arbitrary, we take $\delta = 1 - \frac{2}{k}$ in the above to obtain that $R$ is $\frac{(k+1)(k-1)}{2k(k-2)}\left[k+1 + \sqrt{\frac{k^3 - 3k^2 + 7k - 1}{k-1}}\right]$-Lipschitz. Hence $k_0(l_\infty)$ is bounded above by this quantity. As $k > 2 + \sqrt{2}$ was arbitrary, we thus have:
\begin{align*}
k_0(l_\infty) &\leq \inf_{k > 2 + \sqrt{2}} \frac{(k+1)(2k-1)}{2k(k-2)}\left[k+1 + \sqrt{\frac{2k^3 - k^2 + 8k - 1}{2k-1}}\right] \\ &= 20.4494\cdots < 20.5,
\end{align*}
which completes the proof.
\end{proof}

\noindent We now turn to $c_0$:

\begin{theorem}
    $k_0(c_0) \leq 6 + 4\sqrt{2}$.
\end{theorem}

\begin{proof}
We begin by fixing some $a \in (0, \frac{1}{2}]$ and considering the annulus $D_a := \{z \in B_{c_0}: a \leq ||z|| \leq 1\}$. Define $V_a : D_a \to S_{c_0}$ by $z \mapsto z + Q_{1-||z||}(\frac{1-a}{a}z)$. First note that, coordinate-wise, $z_j$ and $Q_{1-||z||}(\frac{1-a}{a}z)_j$ have the same sign $\forall z \in D_a$. Hence, 
\[
|V_a(z)_j| = |z_j| + \left|Q_{1-||z||}\left(\frac{1-a}{a}z\right)_j\right| = |z_j| + \min\{1 - ||z||, \frac{1-a}{a}|z_j|\}
\]
Now this yields $|V_a(z)_j| \leq |z_j| + 1 - ||z|| \leq 1$ for all $j$, i.e. $||V_a(z)|| \leq 1$, but also, for $j$ s.t. $|z_j| = ||z||$ (all sequences in $c_0$ achieve their norm), 
\[
\frac{1-a}{a}|z_j| = \frac{1-a}{a} ||z|| \geq 1-||z||,
\]
because $||z|| \geq a$, hence $|V_a(z)_j| = |z_j| + 1 - ||z|| = 1$. Thus we see that $V_a$ does in fact have codomain $S_{c_0}$. Furthermore, for $z \in S_{c_0}$, $1-||z|| = 0$, so $V_a(z) = z$. Thus $V_a$ is a retraction, and we have, as a direct consequence of the non-expansiveness of $Q_{1-||z||}$ that $Lip(V_a) \leq 1+ \frac{1-a}{a} = \frac{1}{a}$.
\newline
\newline
Now for $x \in B_{c_0}$, let $q_\lambda(r) := \min\{\frac{1}{2}, \lambda(1-r)\}$ for some parameter $\lambda > 1$ and take $F_\lambda(x) := (q_\lambda(||x||), [\lambda x_1]_{q_\lambda(||x||)}, [\lambda x_2]_{q_\lambda(||x||)}, \cdots) \in c_0$. Consider the map $A_\lambda(x) := x - F_\lambda(x)$. If $||x|| \geq \frac{2\lambda - 1}{2\lambda}$, so that $q_\lambda(||x||) = \lambda(1-||x||)$, we have:
\begin{align*}
||A_\lambda(x)|| \geq ||x|| - ||F_\lambda(x)|| &\geq \frac{2\lambda - 1}{2\lambda} - q_\lambda(||x||) \\ & \geq \frac{2\lambda - 1}{2\lambda} - \frac{1}{2}\\ &= \frac{\lambda - 1}{2\lambda}
\end{align*}
This also holds for $||x|| < \frac{2\lambda - 1}{2\lambda}$. To see this, note that in this case, $q_\lambda(||x||) = \frac{1}{2}$, and suppose for contradiction that $||A_\lambda(x)|| < \frac{\lambda - 1}{2\lambda}$. Then by considering the first term of the sequence $A_\lambda(x)$, we get $|x_1 - \frac{1}{2}| < \frac{\lambda - 1}{2\lambda}$, hence $|x_1| > \frac{1}{2\lambda}$. Thus $|\lambda x_1| > \frac{1}{2}$, so by considering the second term, either $|x_2 - \frac{1}{2}| < \frac{\lambda - 1}{2\lambda}$ or $|x_2 + \frac{1}{2}| < \frac{\lambda - 1}{2\lambda}$. Thus $|x_2| > \frac{1}{2\lambda}$, and iterating this, we obtain $|x_n| > \frac{1}{2\lambda}$ $\forall n \in \mathbb{N}$, contradicting $x \in c_0$.
\newline
\newline
Now this means that the map $Q_1 \circ A_\lambda$ maps into (i.e. has codomain) the annulus $D_{\frac{\lambda - 1}{2\lambda}}$. Furthermore, $q_\lambda$ and hence $F_\lambda$ are $\lambda$-Lipschitz, hence $A_\lambda$ and so $Q_1 \circ A_\lambda$ are $(\lambda+1)$-Lipschitz. Furthermore, if $||x|| = 1$, then $q_\lambda(||x||) = 0$, thus $F_\lambda(x) = 0$, and $A_\lambda(x) = x$. Thus $Q_1 \circ A_\lambda$ is a retraction. Finally we define the retraction $R := V_{\frac{\lambda - 1}{2\lambda}} \circ Q_1 \circ A_\lambda$. By all of the above, this maps $B_{c_0} \to S_{c_0}$ and is a retraction. Also, its Lipschitz constant is at most $\frac{2\lambda(\lambda+1)}{\lambda - 1}$.
\newline
\newline
As $\lambda > 1$ was arbitrary, this means that:
\begin{align*}
    k_0(c_0) \leq \inf_{\lambda > 1} \frac{2\lambda(\lambda+1)}{\lambda-1} = 6+4\sqrt{2},
\end{align*}
as required.
\end{proof}

\begin{remark}
The previously best-known upper bound was $4(2+\sqrt{3}) = 14.928\cdots$, and ours is $6+4\sqrt{2} = 11.656\cdots$.
\end{remark}

\noindent Finally, we turn to the case of $l_1$:

\begin{theorem}
    $k_0(l_1) \leq 4 + \sqrt{13} = 7.6055\cdots$.
\end{theorem}

\begin{proof}
    Let $S: l_1 \to l_1$ be the shift map $(x_1, x_2, \cdots) \mapsto (0, x_1, x_2, \cdots)$. Also, for $0 \leq t < ||x||$, let $C_tx$ be the portion of $x \in l_1$ having $l_1$-mass exactly $t$. More precisely, if $j \in \mathbb{N}$ is the smallest index such that $\sum_{m=1}^j |x_m| > t$, then $(C_tx)_i = x_i$ for $i < j$, $(C_tx)_j = \operatorname{sign}(x_j) \cdot (t - \sum_{m=1}^{j-1} |x_m|)$, and $(C_tx)_i = 0$ for $i > j$. (As edge cases, we take $C_0(0) = 0$, and $C_{||x||}x = x$ $\forall x \in l_1$). Also, we let $(e_n)_{n \geq 1}$ denote the standard basis of $l_1$, as usual.
    \newline
    \newline
    Now let $a := 4 - \sqrt{13}$, so that $K := \frac{3}{a} = 4+\sqrt{13}$. Let $m(r) := \frac{(r-a)(3r-a)}{4a}$ for $r \in [a, 1]$, and observe that $m(r) \geq 0$ in this range. Also, $r - m(r) = \frac{(1-r)(3r+3-8a)}{4a} \geq 0$ (because $a^2 - 8a + 3 = 0$), so $m(r) \leq r$. Now, we define the retraction which we will show to have small Lipschitz constant. We define it for a dense set in $B_{l_1}$ - namely, for $x \in B_{l_1} \cap c_{00}$, let
    \[
    R(x) := \begin{cases}
        \left(1 - \frac{||x||}{a}\right)e_1 + \frac{1}{a}Sx \qquad \text{if $||x|| \leq a$} \\ \frac{C_{m(||x||)}x + S(x - C_{m(||x||)}x)}{||x||} \qquad \text{if $a \leq ||x|| \leq 1$}
    \end{cases}
    \]
    Now, we make a few observations. Firstly, if $||x|| \leq a$, then because $Sx$ is concentrated in $span\{e_2, e_3, \cdots\}$, we have $||R(x)|| = 1 -  \frac{||x||}{a} + \frac{1}{a}||Sx|| = 1 -  \frac{||x||}{a} + \frac{||x||}{a} = 1$. Secondly, for $||x|| \in [a, 1]$, $S(x - C_{m(||x||)}x)$ is concentrated on a set of coordinates disjoint from the support of $ C_{m(||x||)}x$, hence, as $S$ is an isometry, $||R(x)|| = \frac{||C_{m(||x||)}x|| + ||x - C_{m(||x||)}x||}{||x||} = \frac{m(||x||) + ||x|| - m(||x||)}{||x||} = 1$. Finally, if $||x|| = a$, then the first expression for $R$ yields $R(x) = \frac{Sx}{a}$, and because $m(a) = 0$, the second expression yields the same. Thus we have shown that $R$ is well-defined, and does indeed have codomain $S_{l_1}$. Note that continuity of $R$ is immediate from continuity of $m$. 
    \newline
    \newline
    We will show that $R$ is $K$-Lipschitz. First, if $||x||, ||y|| \leq a$, then 
    \[
    ||R(x) - R(y)|| \leq \left|\frac{||x||}{a} - \frac{||y||}{a}\right| + \frac{1}{a}||Sx-Sy|| \leq \frac{2}{a}||x-y|| \leq K||x-y||,
    \]
    so we need only deal with the case $||x||, ||y|| \geq a$ (note that the case $||x|| \leq a, ||y|| > a$ then follows by considering the point $x_* \in aB_{l_1}$ lying on the line between $x$ and $y$, and applying the Lipschitz bound from both cases to obtain
    \[
    ||R(x) - R(y)|| \leq ||R(x) - R(x_*)|| + ||R(x_*) - R(y)|| \leq K(||x-x_*|| + ||x_*-y||) = K||x-y||,
    \]
    the desired conclusion). Suppose from now $||x||, ||y|| \geq a$.
    \newline
    \newline
    Let $\varepsilon_i := \operatorname{sign}(x_i)$. We first deal with the scenario where $x, y$ differ in one coordinate $k$ only, and with $y_k = x_k + h \varepsilon_k$ ($h > 0$). Note that then $||y|| = ||x|| + h$. Now suppose also that the cutoff for $C_{m(||y||)}y$ occurs at the same coordinate $j$ as the cutoff for $C_{m(||x||)}x$. Then:
    \newline
    \begin{itemize}
        \item If $k < j$, the $k^{th}$ coordinates of $C_{m(||y||)}y$ and $C_{m(||x||)}x$ differ by $h$, and the $j^{th}$ coordinates differ by $|m(||y||) - m(||x||) - h|$. Further, the $(j+1)^{th}$ coordinates of $S(y - C_{m(||y||)}y)$ and $S(x - C_{m(||x||)}x)$ differ by the same amount. Taking $\mu := \frac{m(||y||) - m(||x||)}{||y|| - ||x||} = \frac{m(||y||) - m(||x||)}{h}$, this yields that $||G(y) - G(x)|| = h(1 + 2|\mu - 1|)$, where $G(x) := C_{m(||x||)}x + S(x - C_{m(||x||)}x)$.
        \newline
        \item If $k > j$, a similar argument (looking at the $j^{th}$ coordinates of $C_{m(||x||)}x$ and $C_{m(||y||)}y$ and the $(j+1)^{th}$ and $(k+1)^{th}$ coordinates of $S(x - C_{m(||x||)}x)$ and $S(y - C_{m(||y||)}y)$) yields that $||G(x) - G(y)|| = h(2\mu + 1)$.
        \newline
        \item If $k = j$, a similar argument yields $||G(x) - G(y)|| = h(\mu + |1-\mu|)$.
    \end{itemize}
    \noindent
    Now, since $R(x) = \frac{G(x)}{||x||}$, we have that $||y||(R(y) - R(x)) = G(y) - G(x) - hR(x)$. Hence, as $||R(x)|| = 1$, we have:
    \[
    ||R(y) - R(x)|| \leq \frac{||G(x) - G(y)|| + h}{||y||} = \frac{||G(x) - G(y)|| + h}{h} \cdot \frac{h}{||y||}
    \]
    We now look at the three cases above, and note that they respectively yield:
    \newline
    \begin{itemize}
        \item $||R(x) - R(y)|| \leq \frac{h}{||y||}(2 + 2|\mu - 1|)$. If $\mu \geq 1$, this is $\frac{2\mu h}{||y||}$. A direct calculation yields $\mu = \frac{K}{4}(||x|| + ||y||) - 1$, hence $2 + 2\mu = \frac{K}{2}(||x|| + ||y||) \leq K||y||$. So $\frac{2\mu h}{||y||} \leq \frac{h}{||y||} \cdot K||y|| = Kh$. As $h = ||x-y||$, we have $||R(x)-R(y)|| \leq K||x-y||$. If instead $\mu < 1$, we have $||R(x)-R(y)|| \leq \frac{h(4-2\mu)}{||y||}$. But 
        \[
        4-2\mu = 6 - \frac{K}{2}(||x|| + ||y||) = \frac{K}{2}(4a - ||x|| - 3||y||) + K||y|| \leq K||y||,
        \]
        as $||x||, ||y|| \geq a$, which again gives $||R(x) - R(y)|| \leq Kh$.
        \newline
        \item The second case yields 
        \[
        ||R(x) - R(y)|| \leq \frac{h}{||y||}(2\mu + 2) \leq \frac{h}{||y||}\cdot K||y|| = Kh,
        \]
        where the second inequality comes from the estimate for $2\mu + 2$ in the previous bullet point.
        \newline
        \item The third case yields $||R(x) - R(y)|| \leq \frac{h}{||y||} (1+\mu + |1-\mu|)$. If $\mu < 1$, $1+\mu + |1-\mu| = 2 \leq 3 = Ka \leq K||y||$, else $1+\mu + |1-\mu| = 2\mu \leq K||y||$ as before. Either way, $||R(x) - R(y)|| \leq Kh = K||x-y||$ once more.
    \end{itemize}
    In total, we have now shown that the desired Lipschitz bound holds for $||x||, ||y|| \geq a$ whenever $x, y$ differ in one coordinate, by a small enough amount that the cutoff coordinates for $C_{m(||x||)}x$ and $C_{m(||y||)}y$ are the same. We may now extend this to any two $x, y \in B_{l_1} \cap c_{00}$ (with $||x||, ||y|| \geq a$) that differ in one coordinate (by any amount) as follows: consider $R(z)$ for $z$ varying along the line between $x$ and $y$. We may split this line into 3 regions: one where $||z|| \geq a$ and decreases as we traverse the line, one where $||z|| < a$, and one where $||z|| \geq a$ and increases as we traverse the line (some of these regions may be empty). Along the first and third regions, since  $||z||$ and hence $m(||z||)$ vary monotonically, and since $z$ is finitely supported along the line, we observe that the cutoff coordinate changes only finitely many times. Thus, we may apply the above argument piecewise as $z$ traverses the first and third regions of the line, and apply the previously derived Lipschitz bound for $R$ in the second region when $||z||$ is small, incurring only arbitrarily small losses at the finitely many boundary points by continuity of $R$, to obtain that $||R(x) - R(y)|| \leq K||x-y||$, as desired. Note that this piecewise argument works because for any finite collection $0 = t_0 < t_1 < \cdots < t_k = 1$, we have
    \[
    \sum_{i = 1}^k ||(x + t_i(y-x)) - (x + t_{i-1}(y-x))||  = \sum_{i = 1}^k (t_i - t_{i-1})||y-x|| = ||y-x||
    \]
    Next, we extend this Lipschitz bound to all of $B_{l_1} \cap c_{00}$ by remarking that, for any two $x, y \in B_{l_1} \cap c_{00}$, we may continuously transform $x$ into $y$ by varying each coordinate of $x$ until it becomes the corresponding coordinate of $y$, one coordinate at a time for each coordinate where they differ, and applying the above result finitely often (again because both $x$ and $y$ are finitely supported). Lastly, we may extend this map $R$ to a map $R^*$ on all of $B_{l_1}$ by density of $c_{00}$ in $l_1$. This map preserves all properties of $R$ - in particular, $R^*$ is a $K$-Lipschitz retraction on $B_{l_1}$, which completes the proof.
\end{proof}

\section{Disclosure of AI usage}
\noindent OpenAI's ChatGPT 5.6 generated the main ideas and constructions for the proofs in this paper; the author checked, revised and rewrote these proofs, and takes full responsibility for the correctness and presentation of the results.
\newpage
\bibliographystyle{plain}
\bibliography{refs}

\end{document}